\documentclass[runningheads]{llncs}
\usepackage[linesnumbered,commentsnumbered,ruled,vlined]{algorithm2e}

\usepackage{verbatim}
\usepackage{blkarray}

\usepackage[T1]{fontenc}
\usepackage{graphicx}
\usepackage{amsmath}
\usepackage{amssymb}
\usepackage{tikz}
\usetikzlibrary{trees, positioning}
 \usepackage{multirow}

 \newcommand{\x}{\textbf{x}}
 \newcommand{\p}{\textbf{p}}
 \newcommand{\z}{\textbf{z}}

\begin{document}

\title{Low-Memory Numerical Certification}

\titlerunning{Low-Memory Numerical Certification}

 \author{Paul Breiding\inst{1}\orcidID{0000-0003-3747-9185} \and
 Taylor Brysiewicz \inst{2}\orcidID{0000-0003-4272-5934} \and
 David K. Johnson\inst{2}\orcidID{0009-0007-9017-6825}}

\authorrunning{P. Breiding et al.}

\institute{University of Osnabr\"uck, Osnabr\"uck, Lower Saxony, Germany\and 
            Western University, London, Ontario, Canada }

\maketitle             

\vspace{-10pt}
\begin{abstract}
We introduce a low-memory framework for certifying numerical solutions to polynomial systems which uses solution iterators and spatial partitioning trees to reduce memory requirements.  We provide a prototypical algorithm, analyze its complexity, and demonstrate the memory reduction on a large example. 
\end{abstract}

\vspace{-28pt}
\section{Introduction}

Certification methods transform numerical solutions to polynomial systems into rigorously proven mathematical statements. The two most popular methods for performing such symbolic-numeric alchemy 
 are Smale's $\alpha$-theory \cite{blumealpha,HSalpha,smalealpha} and the Krawczyk method \cite{JuliaCertification,Krawczyk,M2Certification}. Each enjoys $O(d)$ space and time complexity to certify $d$ candidate solutions to a system $F$, and the Krawczyk method has been parallelized in \cite{HC}. However, the existing implementations require simultaneously holding all solutions in memory, a bottleneck for $d \gg 0$. Our method circumvents this memory issue by partitioning the solution space $\mathbb{C}^n$ into parts and certifying the candidates, represented by a solution iterator \cite{HomotopyIterators,MonodromyCoordinates},  
  in each part.
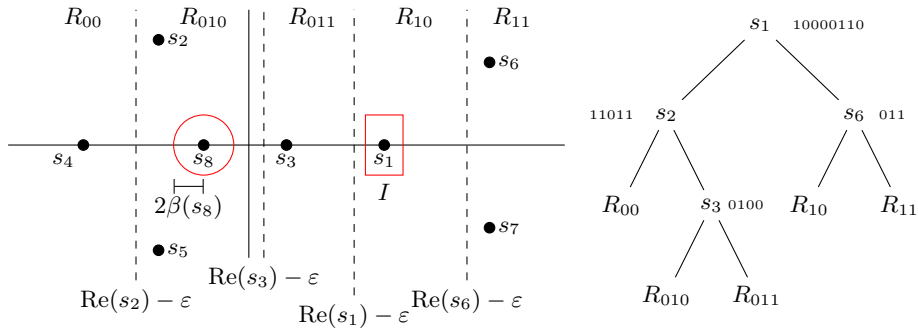
\begin{figure}[!htpb]
\begin{center}
\begin{tikzpicture}[
    every node/.style={font=\small},
    level 1/.style={sibling distance=2.5cm, level distance=1.2cm},
    level 2/.style={sibling distance=1.2cm, level distance=1.2cm},
    level 3/.style={sibling distance=1.2cm, level distance=1.2cm},
]

\begin{scope}[xshift=-5cm]

  \draw[-] (-3.2,0) -- (4.2,0);
  \draw[-] (0,-1.5) -- (0,1.8);

  \draw[dashed] (-1.5,-1.8) node[below] {$\textrm{Re}(s_2)-\varepsilon$} -- (-1.5, 1.8) ;
  \draw[dashed] ( 0.2,-1.5) node[below] {$\textrm{Re}(s_3)-\varepsilon$} -- ( 0.2, 1.8) ;
  \draw[dashed] ( 1.4,-2.) node[below] {$\textrm{Re}(s_1)-\varepsilon$} -- ( 1.4, 1.8) ;
  \draw[dashed] ( 2.9,-1.8) node[below] {$\textrm{Re}(s_6)-\varepsilon$} -- ( 2.9, 1.8) ;

  \node at (-2.2, 1.7) {$R_{00}$};
  \node at (-0.6, 1.7) {$R_{010}$};
  \node at ( 0.85, 1.7) {$R_{011}$};
  \node at ( 2.2,  1.7) {$R_{10}$};
  \node at ( 3.5,  1.7) {$R_{11}$};

  \filldraw (-2.2,  0  ) circle (2pt) node[below left]  {$s_4$};
  \filldraw (-1.2,  1.4) circle (2pt) node[right]       {$s_2$};
  \filldraw (-1.2, -1.4) circle (2pt) node[right]       {$s_5$};
  \filldraw (-0.6,  0  ) circle (2pt) node[below]       {$s_8$};
  \filldraw ( 0.5,  0  ) circle (2pt) node[below]       {$s_3$};
  \filldraw ( 1.8,  0  ) circle (2pt) node[below]       {$s_1$};
  \filldraw ( 3.2,  1.1) circle (2pt) node[right]       {$s_6$};
  \filldraw ( 3.2, -1.1) circle (2pt) node[right]       {$s_7$};
  
  \draw[red] (-0.6, 0) circle (0.4cm);
  \draw[|-|] (-1.0,-0.55) -- (-0.6,-0.55)
    node[midway, below, font=\small] {$2\beta(s_8)$};

  \draw[red] (1.8-0.25, -0.4) rectangle (1.8+0.25, 0.4);
  \node[black, below=0.4cm] at (1.8, 0) {$I$};

\end{scope}

\begin{scope}[xshift=1.8cm, yshift=1.6cm]

 \node (r1) {$s_1$}
    child { node (r2) {$s_2$}
      child { node {$R_{00}$} }
      child { node (r4) {$s_3$}
        child { node {$R_{010}$} }
        child { node {$R_{011}$} }
      }
    }
    child { node (r3) {$s_6$}
      child { node {$R_{10}$} }
      child { node {$R_{11}$} }
    };

  \node[right=0.05cm of r1, font=\ttfamily\tiny] {$10000110$};
  \node[left=0.05cm of r2, font=\ttfamily\tiny] {\hspace{6pt}$11011$};
  \node[right=0.05cm of r3, font=\ttfamily\tiny] {\hspace{-3pt}$011$};
  \node[right=0.05cm of r4, font=\ttfamily\tiny] { \hspace{-6pt}$0100$};

\end{scope}

\end{tikzpicture}
\end{center}
\caption{
(Left) A partition $\bigsqcup_{\ell \in \textrm{leaves}(T)} R_\ell$ of $\mathbb{C}$ with candidate solutions $S=\{s_i\}_{i=1}^8$ of a real polynomial system. The red regions indicate the certified regions of the Krawczyk method and $\alpha$-certification for $s_1$ and $s_8$ respectively. 
(Right) A binary tree encoding the partition. Each node $\nu$ represents a region $R_{\nu}$, which is then refined  into two smaller regions by the outcomes of  $\textrm{Re}(\z)<\textrm{Re}(s_i)-\varepsilon$ for \textit{split} $s_i$ decorating $\nu$. The subset of these outcomes on   $R_{\nu} \cap S$ are cached in the corresponding bitvector. 
}
\label{fig:partition1}
\end{figure}

By design, our algorithm can be applied to the datatype of a \textit{solution iterator}. A solution iterator represents a solution set $S \subseteq \mathbb{C}^n$ by holding a single solution in memory and promising the ability to obtain the \texttt{next} solution. Introduced in \cite{MonodromyCoordinates} and expanded upon in \cite{HomotopyIterators}, solution iterators reduce the memory complexity of many existing algorithms through their lazy-evaluation approach.
In the present work, we extend the  applicability  of solution iterators to certification techniques. 
 
After fixing some number $k \in \mathbb{N}$, our main idea is to partition $S$ into parts of size no larger than $k$ and then collect and certify each of these $m$ parts, one at a time. Practically, this approach requires representing a partition of $S$ in (1) a computationally efficient manner such that (2) one may certify there is no overlap between the certified regions in distinct parts. A \textit{binary spatial partitioning} (BSP) tree, as shown in Figure \ref{fig:partition1} (right), addresses both concerns. 

Algorithm \ref{alg:buildTree} builds a BSP tree from a solution iterator which is, heuristically, balanced; in particular  $m\approx d/k$. The user may  sacrifice some storage to make this tree much more time-efficient via \textit{bitmasks}.  Algorithm \ref{alg:iteratorcertify} takes the tree along with the solution iterator and certifies all parts. Table \ref{tab:complexitysummary} summarizes the complexity of our main algorithm, Algorithm \ref{alg:mainalgorithm}, which combines these subroutines.

\vspace{-10pt}
\begin{table}
\begin{center}
\begin{tabular}{{|l||l||c|c|c|c|}}
\hline
 & \quad \quad \quad \quad \, Space & \multicolumn{3}{c|}{Time in \# evaluations }\\ \hline \hline
Bitmasks &  & Bit lookups & \,\,\,$<$ \,\,\, & \texttt{next}   \\ \hline 
Krawczyk & $d\log_2(m) + 384nk$ 
& \multirow{2}{*}{\shortstack{$O(dm\log(m)$\\$+\,d\log(m)^3)$}} 
& \multirow{2}{*}{$O(d\log(m))$} 
& \multirow{2}{*}{$O(d\log(m))$}\\
$\alpha$-certification & $d\log_2(m) + 384nk+128k$ & & & \\
\hline
\hline
No Bitmasks & &Bit lookups & $<$  & \texttt{next}  \\ \hline \hline
Krawczyk & $64(m-1) + 384nk$ 
&\multirow{2}{*}{$0$}  
&\multirow{2}{*}{$O(dm\log(m))$}  
&\multirow{2}{*}{$O(dm)$} \\
$\alpha$-certification & $64(m-1) + 384nk+128k$ & & & \\
\hline
\end{tabular}
\end{center}
\caption{Space and time complexity summary of our main algorithm, Algorithm \ref{alg:mainalgorithm}.}
\label{tab:complexitysummary}
\end{table}

\vspace{-35pt}

\section{Certification background}
Consider a polynomial system
\[
F: f_1(\x) = \cdots = f_n(\x) = 0 \quad \text{ where } \quad f_1,\ldots,f_n \in \mathbb{Q}[x_1,\ldots,x_n] = \mathbb{Q}[\x] 
\]
and a set 
$S = \{s_1,\ldots,s_d\} \subseteq \mathbb{C}^n$
of floating point approximations of roots of $F$ called \textit{candidates}. 
There are two main algorithms which convert $S$ into some mathematical result: Smale's $\alpha$-theory and the Krawczyk method. Both methods deduce facts about the Newton operator
$
N_F(\x) = \x-\textrm{Jac}_F(\x)^{-1} \cdot F(\x).
$

\paragraph{Certification by $\alpha$-theory}~

\noindent In \textit{Smale's $\alpha$-theory}, one certifies a candidate $s \in \mathbb{C}^n$ by first computing some constants $\alpha(F,s), \beta(F,s)$ and $\gamma(F,s)$  (see \cite{HSalpha} for details).  The foundational result comes from \cite[Page 160]{blumealpha}. 
\begin{theorem}\label{thm_alpha}
If $\alpha(F,s)<\frac{13-3\sqrt{17}}{4} \approx 0.157971$ then $s$ converges quadratically to a true root $z$ of $F$ under repeated application of Newton's method. Moreover, the true root $z$ is at most $2\beta(F,s)$ from $s$. 
\end{theorem}
Practically, one either uses a rational approximation of $s$ to compute exactly, or interval arithmetic methods \cite{AlphaInterval}.

\paragraph{Interval certification}~

\noindent The second certification paradigm is \textit{the Krawczyk method} \cite{Krawczyk} which shows that the Newton operator is a contraction  on an interval box $I \subseteq \mathbb{C}^n \cong \mathbb{R}^{2n}$ containing $s$. It verifies this using interval arithmetic (see \cite{JuliaCertification} for details). Successfully showing that $N_F$ is a contraction on $I$ proves that $I$ contains a unique root by Banach's fixed point theorem \cite{Banach}.
\begin{theorem}\label{thm_interval}
Let $I \subset \mathbb{C}^n \cong \mathbb{R}^{2n}$ be an interval box. If $N_F$ is a contraction on $I$ then $I$ contains a unique root of $F$.
\end{theorem}

\paragraph{Parallel certification using distance to a reference point}~

\noindent Both $\alpha$-certification and the Krawczyk method have three core capabilities:

\vspace{2pt}
\begin{tabular}{lll}
\multicolumn{3}{l}{1. Associate to an approximation $s$ a true root $z$ of $F$.}\\[0.25em]
&\; ($\alpha$-theory) $s \to z$ under $N_F$ & \quad \hspace{-10pt} (Krawczyk) $I \to z$ under $N_F$\\[0.25em]
\multicolumn{3}{l}{2. Provide a \textit{certificate region} containing $s$ and the true root $z$ of $F$.}\\[0.25em]
&\; ($\alpha$-theory) $B_{2\beta(F,s)}(s)$ & \quad \hspace{-10pt} (Krawczyk) $I$ \\[0.25em]
\multicolumn{3}{l}{3. Check if two certified approximations $s,s'$ correspond to distinct roots.}\\
&\; ($\alpha$-theory)  $||s-s'|| > 2(\beta(F,s)+\beta(F,s'))$ & \quad\hspace{-10pt} (Krawczyk) $I \cap I' = \emptyset$
\end{tabular}
\vspace{2pt}

\noindent Certifying a single point involves Item $1$ in the above list. The difficulty of parallel or low-memory certification, however, is the need to show that the true roots obtained from Item $1$ are distinct. Doing so involves Item $3$ which, \textit{prima facie}, requires $O(d^2)$ comparisons and obstructs trivial parallelizability. Alternatively, one can compute (in parallel) the distance intervals that the certificate regions, obtained in Item $2$, inhabit with respect to some fixed reference point $\textbf{z} \in \mathbb{C}^n$. Then, one must show only that the candidates whose intervals intersect are distinct, which may be done in parallel. This technique is implemented in \cite{HC}.  However, incorporating it into the lazy-evaluation framework of solution iterators still requires significant memory costs (see Example~\ref{ex:example}).

\section{Binary spatial partitioning trees}
Other than solution iterators, the main datatype involved in our algorithms is that of a \textit{binary spatial partitioning} (BSP) tree. Our version of a BSP tree is simplistic, as we do not require nor benefit from more intricate constructions, such as $kd$-trees \cite[Chapter 5]{deBerg2008}. We seek only to induce a spatial partition on a set $S$ of $d$ solution candidates such that no part has size larger than $k$. 

Define a partial order on $\mathbb{C}^n$ using just the real part of the first coordinate: 
\[
\textbf{z}=(z_1,\ldots,z_n) <(w_1,\ldots,w_n)=\textbf{w} \iff \textrm{Re}(z_1)<\textrm{Re}(w_1).
\]
Note that conjugate points $\z$ and $\overline{\z}$ are incomparable.  For any $\z \in \mathbb{C}^n$ we write 
\[
H_{\z,-} = \{\p \in \mathbb{C}^n \mid \p<\z \} \quad \quad \quad 
H_{\z,+} = \{\p \in \mathbb{C}^n \mid \p\not<\z  \}
\]
for the left open halfspace and right closed halfspace defined by $\z$, which is called the \textit{splitting point} or \textit{split}.  In practice, we shift the split by some tolerance $\varepsilon~>~0$ and define $
\mathcal H_{\z,-}  = H_{\z-\varepsilon \cdot e_1,-}$ and $
\mathcal H_{\z,+}  = H_{\z-\varepsilon \cdot e_1,+}$.
 This is necessary since these halfspaces must be later shown (step \texttt{4} of Algorithm \ref{alg:iteratorcertify}) to contain certificate regions of split points, as shown in the left of  Figure~\ref{fig:partition1}.
 
Let $T$ be a binary tree whose internal nodes represent splits in $\mathbb{C}^n$ so that labels of left (resp. right) descendants of a split $\z$ belong to $\mathcal H_{\z,-}$ (resp. $\mathcal H_{\z,+}$). Every node $\nu$, including the leaves, of such a tree thus represents a region obtained via the intersection of the halfspaces indexed by its split ancestors:
\[
R_{\nu} := \left(\bigcap_{\nu \prec\z} \mathcal H_{\z,-} \right) \cap \left(\bigcap_{\nu \succ\z} \mathcal H_{\z,+} \right).
\]
Here $\prec$ indicates ``left descendant'' and $\succ$ indicates ``right descendant''. Such a tree is called a \textit{binary spatial partitioning} (BSP) tree since the leaf regions of $T$ partition $\mathbb{C}^n$.  
The following algorithmic structure builds a BSP tree iteratively.
\begin{algorithm}\label{alg:buildTree}
\caption{Build BSP Tree}
\SetKwInOut{Input}{Input}
\SetKwInOut{Output}{Output}
\Input{A set $S\subseteq\mathbb{C}^n$ of size $d$\\
	   A number $k \in \mathbb{N}$ bounding the size of each part}
\Output{A BSP tree $T$ whose leaf regions each contain at most $k$ points of $S$\\
}
Choose $\z_1\in \mathbb{C}^n$.\\
Initialize $T$ with root (depth $0$) labeled with the split $\z_1$  \,\,\,\, \quad \quad \quad \quad \quad \quad \,\, (implicitly branching children representing $\mathcal H_{\z_1,-}$ and $\mathcal H_{\z_1,+}$). \\
\While{there exists a leaf $\ell$ of $T$ such that $|R_{\ell}\cap S|\geq k$}{
Choose  $\z\in R_{\ell}$ and assign $\z$ to $\ell$ (implicitly branching children).\\
}
\Return{T}
\end{algorithm}\\
Algorithm \ref{alg:buildTree} needs a method for choosing splits in steps $\texttt{1}$ and $\texttt{4}$. Choices include 
\begin{enumerate}
\item (Median) Choose the median of the real part of the first coordinate of $R_{\nu} \cap S$.
\item (Mean) Choose the mean of the real part of the first coordinate of $R_{\nu} \cap S$.
\item (Random) Choose a uniform random point in $R_{\nu} \cap S$.
\end{enumerate}
The \textit{Median} method is used in many algorithms \cite{deBerg2008} as it guarantees that the children represent subsets of $S$ of almost equal sizes at each step, inducing a tree which is \textit{balanced}. The \textit{Mean} method will produce an approximately balanced tree provided the points $S$ are spatially well-behaved. Similarly, for large $d$, the \textit{Random} method will produce an approximately balanced tree. For ease of analysis, we henceforth assume that any of these methods produces a balanced BSP tree. In particular, we assume that there are $m\approx d/k$ leaves, each indexing a region with $\approx k$ points of $S$. Thus, the tree has height $\log_2(m)$ and computing the leaf region containing some $\z \in \mathbb{C}^n$ costs $\log_2(m)$ evaluations of $<$.

\section{Implementing BSP trees over iterators}
A \textit{solution iterator} $\mathcal I$ is a data structure which represents the $d$ solutions $S\subseteq \mathbb{C}^n$ to a polynomial system. It holds a single solution $s_i$  in memory and provides the user with access to a $\texttt{next}$ function that exchanges the solution $s_i$ in memory with $s_{i+1}$. In  \cite{HomotopyIterators} and \cite{MonodromyCoordinates}, the $\texttt{next}$ function requires some amount of numerical \textit{path-tracking}, the core algorithm of numerical algebraic geometry \cite{NAG}. Path-tracking is much more expensive than accessing a point in a list.

Algorithm \ref{alg:buildTree} can be performed using an iterator $\mathcal I$ for $S$ as input, instead of the solution set $S$ itself. Storing the real part of the first coordinates of the splits on each of the $m-1$ internal nodes of the tree allows one to query which leaf region contains a given point. Finding the elements $R_{\nu} \cap S$ in a region at depth $\delta$ in step~\texttt{3} requires $d$ \texttt{next} calls and $d\delta$ evaluations of $<$. Doing this for all $2m-2$ non-root regions requires $2dm-2d$ \texttt{next} calls and $\sum_{\delta=1}^{\log_2(m)} 2^{\delta}d\delta=2d(m\log_2(m)-m+1)$ many $<$-evaluations.

 Algorithm \ref{alg:buildTree} is iterative and so many of these $<$-evaluations are recomputed several times. As such, we may choose to cache the $<$-evaluation results in a \texttt{bitmask} $b_\nu$ of length $R_{\nu} \cap S$ so that the coordinate $(b_{\nu})_j$ agrees with the boolean value $\z_\nu<(R_{\nu} \cap S)_j$ where $\z_\nu$ is the split of $\nu$.  See the right of Figure \ref{fig:partition1} for an example of such a bitmasked BSP tree.

When bitmasking the BSP tree, all of the bitmasks at depth $\delta$ may be constructed using a single pass of $\mathcal I$. Doing so costs one \texttt{next}-call, $\delta$-many bit lookups, and one $<$-evaluation for each of the $d$ solutions. Summing over the $\log_2(m)$-many levels constructed in Algorithm \ref{alg:buildTree} proves the following theorem.

\begin{theorem}\label{thm:bspcomplexity}
The time and space required to construct a balanced BSP tree via Algorithm \ref{alg:buildTree} is given in the following table.
\begin{center}
\begin{tabular}[!htpb]{|l||c||c|c|c|}
\hline
 & Space & \multicolumn{3}{c|}{Time} \\  \hline \hline 
 &  & Bit lookups & $<$ & \texttt{next}  \\ \hline 
Bitmasked & $d\log_2(m)$ & $O(d\log_2(m)^3)$
& $d\log_2(m)$ & $d\log_2(m)$ \\
Not Bitmasked\,\, &\,\, $64(m-1)$ \,\,& $0$ & \,\,$O(dm\log_2(m))$ \,\,& $O(dm)$ \\
\hline
\end{tabular}
\end{center}
\end{theorem}
Equipped with bitmasks on each internal node of the BSP tree, one may construct an iterator $\mathcal I_{\ell}$ for each leaf $\ell$ whose \texttt{next} function skips over the solutions not in that leaf. This can often be performed without the need of several  \texttt{next} calls of $\mathcal I$, as is the case for homotopy iterators \cite{HomotopyIterators}. 
\begin{corollary}
\label{cor:collectioncosts}
Let $T$ be a balanced bitmasked BSP tree. The space and time complexity to collect points in one leaf region is given in the following table. 
\begin{center}
\begin{tabular}[!htpb]{|l||c||c|c|c|}
\hline
 & Space & \multicolumn{3}{c|}{Time} \\  \hline \hline 
 &  & Bit lookups & $<$ & \texttt{next}  \\ \hline 
Bitmasked & $64(2n)k$ & $d\log_2(m)$ & $0$ & $k$ \\
Not Bitmasked\,\, &\,\, $64(2n)k$ \,\,& $0$ & \,\,$d\log_2(m)$ \,\,& $d$ \\
\hline
\end{tabular}
\end{center}
\end{corollary}
\begin{remark}
One may hope to improve upon the time or space complexity of our tree structure by evaluating a function which is not binary, which would partition the solution space into many parts via one evaluation. For example, one may use a many-valued function, like  partitioning the solutions into the $4^{2n}$ quadrants of $\mathbb{R}^{2n} \cong \mathbb{C}^n$. This saves no memory as caching the results requires moving beyond bit-vectors in such a way that precisely cancels out the memory benefits of storing fewer function caches. See Figure \ref{fig:equivalentmemory} for an example. 
Another drawback is that one expects to produce a less balanced tree this way.

On the other hand, this approach can save time, since the number of \texttt{next} calls is proportional to the height of the tree. One strategy is to choose splits, \textit{a priori} unrelated to $S$, in hopes that they partition $S$ into sufficiently small parts in order to bypass the iterative complexity of Algorithm \ref{alg:buildTree}. 
\end{remark} 
\vspace{-10pt}
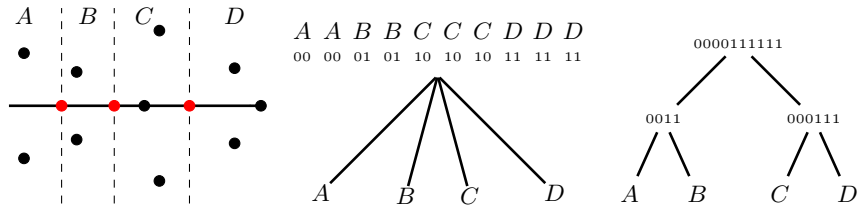
\begin{figure}[!htpb]
\begin{center}
\begin{tikzpicture}[font=\small]
\begin{scope}[xshift=0cm]
  \draw[thick] (1.5,0) -- (1.5,0); 
 

  \node at (-1.5, 1.2)  {$A$};
  \node at (-.65, 1.2)  {$B$};
  \node at (0.1, 1.2)  {$C$};
  \node at (1.3, 1.2)  {$D$};
  \draw[line width=1pt, black] (-1.7,0) -- ( 1.7,0);
 
  \foreach \x in {-1.0, -0.3, 0.7}{
    \draw[dashed] (\x,-1.3) -- (\x, 1.3);
  }
 
  \filldraw[red] (-1.0,  0) circle (2pt);
  \filldraw[red] (-0.3, -0) circle (2pt);
  \filldraw[red] (0.7, -0) circle (2pt);

  \filldraw[black] (-1.5,  0.7) circle (2pt);
  \filldraw[black] (-1.5, -0.7) circle (2pt);
 
  \filldraw[black] (-0.8,  0.45) circle (2pt);
  \filldraw[black] (-0.8, -0.45) circle (2pt);
 
  \filldraw[black] ( 0.3,  1.0) circle (2pt);
  \filldraw[black] ( 0.3, -1.0) circle (2pt);
  \filldraw[black] ( 0.1, 0) circle (2pt);
 
  \filldraw[black] ( 1.3,  0.5) circle (2pt);
  \filldraw[black] ( 1.3, -0.5) circle (2pt);
 
  \filldraw[black] ( 1.65, 0) circle (2pt); 
\end{scope}
 
 
\begin{scope}[xshift=4cm, yshift=-1.2cm]
 
  \newcommand{\lbl}[3]{
    \node at (#1, 2.2) {#2};
    \node at (#1, 1.85) {\tiny #3};
  }
 
  \node at (-1.8, 2.2)  {$A$};  \node at (-1.8, 1.85) {\tiny 00};
  \node at (-1.4, 2.2)  {$A$};  \node at (-1.4, 1.85) {\tiny 00};
  \node at (-1.0, 2.2)  {$B$};  \node at (-1.0, 1.85) {\tiny 01};
  \node at (-0.6, 2.2)  {$B$};  \node at (-0.6, 1.85) {\tiny 01};
  \node at (-0.2, 2.2)  {$C$};  \node at (-0.2, 1.85) {\tiny 10};
  \node at ( 0.2, 2.2)  {$C$};  \node at ( 0.2, 1.85) {\tiny 10};
  \node at ( 0.6, 2.2)  {$C$};  \node at ( 0.6, 1.85) {\tiny 10};
  \node at ( 1.0, 2.2)  {$D$};  \node at ( 1.0, 1.85) {\tiny 11};
  \node at ( 1.4, 2.2)  {$D$};  \node at ( 1.4, 1.85) {\tiny 11};
  \node at ( 1.8, 2.2)  {$D$};  \node at ( 1.8, 1.85) {\tiny 11};
 
   \node[inner sep=0pt] (root4) at (0.0, 1.6) {}
    child[grow=225, level distance=2.2cm, line width=1pt,
          edge from parent/.style={draw, line width=1pt, black}] { node[inner sep=0pt] {$A$} }
    child[grow=255, level distance=1.65cm, line width=1pt,
          edge from parent/.style={draw, line width=1pt, black}]            { node[inner sep=0pt] {$B$} }
    child[grow=285, level distance=1.65cm, line width=1pt,
          edge from parent/.style={draw, line width=1pt, black}]           { node[inner sep=0pt] {$C$} }
    child[grow=315, level distance=2.2cm, line width=1pt,
          edge from parent/.style={draw, line width=1pt, black}] { node[inner sep=0pt] {$D$} }
  ;
 
\end{scope}
 
 
 
\begin{scope}[xshift=7.0cm, yshift=-0.67cm]
 
  \node  (rootB) at (1.0, 1.5) {\tiny 0000111111}
    [grow=down,
     every child/.style={line width=1pt, edge from parent/.style={draw, line width=1pt}},
     level 1/.style={sibling distance=2.0cm, level distance=1.0cm},
     level 2/.style={sibling distance=0.9cm, level distance=1.0cm}]
    child { node (leftB) {\tiny 0011}
      child { node  {$A$} }
      child { node  {$B$} }
    }
    child { node  (rightB) {\tiny 000111}
      child { node  {$C$} }
      child { node  {$D$} }
    }
  ;
\end{scope}
\end{tikzpicture}
\end{center}
\caption{(Left) 10 black points partitioned via the four regions created by three red points. (Middle) The $4$-ary tree with a $4$-ary vector of length $10$ indicating the regions of each point, requiring $10\cdot \log_2(4) = 20$ bits. (Right) The BSP tree obtained by partitioning via the $3$ splits, in order, decorated with bit vectors also requiring $20$ bits. }
\label{fig:equivalentmemory}
\end{figure}
\begin{remark}
Using iterators, the \textit{Median} method for choosing splits in Algorithm~\ref{alg:buildTree} seems infeasible, but the \textit{Mean} method can be performed easily by accumulating the sum of the points in each leaf at each depth during the single iteration over $\mathcal I$. This is one of the observations in \cite{MonodromyCoordinates}, namely that the trace of a solution set may be computed using an iterator with low-memory cost. 
\end{remark}

\section{Certifying solution iterators}	
We arrive at the main idea for certifying solution candidates $S \subseteq \mathbb{C}^n$, to a polynomial system $F$, which are represented by an iterator: using a heuristically balanced BSP tree, certify the elements of $S$ one leaf region at a time. This approach incurs the additional step of verifying that the certificate regions are contained in the interiors of the correct leaf regions. Performing this verification guarantees that the true roots associated to the candidate solutions in one region cannot coincide with the roots associated to a candidate solution in another region.  For the Krawczyk method, this region is the interval box $I$ surrounding the candidate $s$, and for $\alpha$-certification, this region is the ball of radius $2\beta(F,s)$ centered at $s$. We make this precise in the following algorithm.

\begin{algorithm}\label{alg:iteratorcertify}
\caption{Certify solution iterator}
\SetKwInOut{Input}{Input}
\SetKwInOut{Output}{Output}
\Input{A solution iterator $\mathcal I$ for $S$\\
A balanced BSP tree $T$ with respect to $S$}
\Output{A number of distinct certified complex and real numerical solutions}
Initialize $N(\mathbb{C}) = 0$ and $N(\mathbb{R})=0$\\
\For{each leaf $\ell$ in $T$}{
Collect $R_{\ell} \cap S$ in memory. \\
Certify the number $N_{\ell}(\mathbb{C})$ (resp. $N_{\ell}(\mathbb{R})$) of distinct complex (resp. real) candidate solutions in $R_{\ell}\cap S$ using $\alpha$-theory or the Krawczyk method whose certificate regions are contained in $R_{\ell}$.\\
Increment $N(\mathbb{C}) += N_{\ell}(\mathbb{C})$ and $N(\mathbb{R}) += N_\ell(\mathbb{R}))$. 
} 
\Return{$N(\mathbb{C})$ and $N(\mathbb{R})$.}
\end{algorithm}

\begin{theorem}
\label{thm:complexity_certify_IT}
The space and time complexity of Algorithm \ref{alg:iteratorcertify} is  given in the following table.

\begin{center}
\begin{tabular}{{|l||l||c|c|c|c|}}
\hline
 & Space & \multicolumn{4}{c|}{Time} \\ \hline \hline
With bitmasked $T$ &  & Bit lookups & $<$  & \texttt{next} & \texttt{certify} \\ \hline 
Krawczyk & $384nk$ & \multirow{2}{*}{$dm\log_2(m)$} & \multirow{2}{*}{$2d$} &  \multirow{2}{*}{$d$}&\multirow{2}{*}{$d$}\\
$\alpha$-certification & $384nk+128k$ & & & &\\
\hline
\hline
No Bitmasking & & Bit lookups & $<$ & \texttt{next} & \texttt{certify} \\ \hline \hline
Krawczyk & $384nk$ & \multirow{2}{*}{$0$}& \multirow{2}{*}{$dm\log_2(m)$} &\multirow{2}{*}{$dm$} & \multirow{2}{*}{$d$} \\
$\alpha$-certification & $384nk+128k$ & & & &  \\
\hline
\end{tabular}
\end{center}
\end{theorem}
\begin{proof}
Both the Krawczyk method and $\alpha$-certification (with interval arithmetic) require an interval box for $k$ solutions upon collection. The interval box involves $2n$-many complex numbers, and the solution uses $n$ complex numbers. Since a complex float costs $128$ bits, and we are collecting at most $k$ solutions, the space complexity follows for the Krawczyk method. For $\alpha$-certification, we require another $2k$ real numbers for the $k$-many values of $\beta$, represented via  intervals. The time complexities follow from Corollary \ref{cor:collectioncosts}, multiplied by the number $m$ of leaves (note that $mk=d$) along with the observation that certification region checks cost two $<$-evaluations each. For the number of \texttt{certify} calls, clearly, each solution must be certified once. \hfill $\square$
\end{proof}

\begin{algorithm}\label{alg:mainalgorithm}
\caption{Main Algorithm}
\SetKwInOut{Input}{Input}
\SetKwInOut{Output}{Output}
\Input{A solution iterator $\mathcal I$ for the $d$ candidate approximate solutions $S$ of a square polynomial system $F$ defined over $\mathbb{Q}$. \\
A maximum part size $k$}
\Output{A number of distinct certified complex and real numerical solutions}
Construct a BSP tree  $T$ for $\mathcal I$ using Algorithm \ref{alg:buildTree} with some tolerance $\varepsilon$.\\
\Return{Algorithm \ref{alg:iteratorcertify} called on $(\mathcal I,T)$.}
\end{algorithm}

\begin{theorem}
\label{thm:memorycost}
The cost of Algorithm~\ref{alg:mainalgorithm} is given by Table \ref{tab:complexitysummary}.
\end{theorem}
\begin{proof}
Table \ref{tab:complexitysummary} tabulates the sum of the complexity of its two subroutines: building the tree (Algorithm \ref{alg:buildTree}) and certifying an iterator given a tree (Algorithm \ref{alg:iteratorcertify}). Thus, it is the sum of the tables given in Theorem \ref{thm:bspcomplexity} and Theorem \ref{thm:complexity_certify_IT}. \hfill $\square$
\end{proof}

Often, one has \textit{a priori} knowledge of $d$ and $n$ when trying to certify a large polynomial system. A natural task is to find the $k^*=dm^*$ which minimizes the space requirements of Algorithm \ref{alg:mainalgorithm}. The next result follows from basic calculus. 

\begin{theorem}
The memory cost of Algorithm \ref{alg:mainalgorithm}, using bitmasks and the Krawczyk method for certification, is minimized at $m^* = 384n\ln(2)$ to be $d\log_2(384n\ln(2))$ $+d\log_2(e)$.  Without bitmasks, the memory cost is minimized at $m^* = \sqrt{d/6n}$ to be  $128\sqrt{6nd}-64$.
\end{theorem}

\begin{example}
\label{ex:example}
There are $d = 666,841,088$ quadric surfaces tangent to nine general quadrics in $\mathbb{C}^3$, each of which is identified by the $n=10$ coefficients of a degree two polynomial in three variables (see \cite{TangentQuadrics,Schubert}). Applying our technique for the Krawczyk method with bitmasking, using parts of size no larger than $k^* \approx 250,533$, 
has a memory cost of $\approx 0.9952$ 
gigabytes. One must call \texttt{next} $d\log_2(m^*) \approx 10,019,198,441$ times. 
Although the memory cost is only around $2.61$ megabytes without bitmasking, one must call \texttt{next} approximately 
$dm \approx 22,111,804,214,194$
times. 

Next, we approximate the memory cost of the parallel approach in Section 2. Computing an interval distance to a fixed reference point requires $2d$ many $64$-bit floats, and so we compute that $\frac{2\cdot 666841088 \cdot 64}{8\cdot 1024^3} \approx 9.93$ 
gigabytes are required.

Finally, we remark on the case in which we can predict splitting points of a balanced BSP tree for $S$ which partitions the set $S$ into parts of size at most $k$. The cost of certifying this way is dominated in time by $d$ calls to $\texttt{next}$ and \texttt{certify}, and uses only $384nk+64(m-1) = \frac{3840k^2-64k+42677829632}{k}$ bits of memory. This is minimized at $k^*\approx 3333$ to be approximately $3.05$ megabytes.
 \hfill $\diamond$

With an iterator version of certification in hand, the difficulty in executing large certification tasks now lies in the iterator construction. The current implementation of \textit{homotopy iterators}, as described in  \cite{HomotopyIterators}, uses the classical start systems. In the case of our example, the total degree system has too many start solutions to reliably use an iterator, and the polyhedral homotopy iterator requires a prohibitively expensive mixed volume computation. Monodromy coordinates, as described in \cite{MonodromyCoordinates}, could make the problem feasible, but have yet to be implemented.
\end{example}
A version of iterator certification, based on the ideas in this paper, has been implemented in the \texttt{julia} \cite{julia}  package  \texttt{HomotopyContinuation.jl} \cite{HC}.

\section*{Acknowledgements}
PB was supported by DFG, German Research Foundation – Projektnummer 445466444. 
TB and DKJ were supported by NSERC discovery grant RGPIN-2023-03551. The authors are grateful to Michael Burr for his suggestions regarding BSP trees.

\bibliographystyle{plain}   
\bibliography{mybib}   

\end{document}